\documentclass[12pt,reqno]{amsart}

\usepackage{nameref}
\usepackage{upgreek}     
\usepackage{enumitem}
\usepackage{graphicx}
\usepackage{mathrsfs}
\usepackage{amsthm}
\usepackage{amssymb}
\usepackage{amsmath,amsthm,amsfonts}
\usepackage{tikz-cd}
\usepackage{mathabx}
\usepackage{hhline}
\usepackage{textcomp}
\usepackage{amsmath,amscd}
\usepackage[all,cmtip]{xy}
\usepackage{mathtools}
\usepackage[margin=1in]{geometry}
\usepackage{lipsum}
\usepackage{extarrows}
\usepackage{upgreek}
\usepackage{bbm}
\usepackage[displaymath]{lineno}
\usepackage[singlespacing]{setspace}%
\usepackage[colorlinks=true,breaklinks=true,bookmarks=true,urlcolor=blue,
citecolor=blue,linkcolor=blue,bookmarksopen=false,draft=false]{hyperref}

\providecommand{\U}[1]{\protect\rule{.1in}{.1in}}

\makeatletter
\let\orgdescriptionlabel\descriptionlabel
\renewcommand*{\descriptionlabel}[1]{%
	\let\orglabel\label
	\let\label\@gobble
	\phantomsection
	\edef\@currentlabel{#1}%
	\let\label\orglabel
	\orgdescriptionlabel{#1}%
}
\makeatother

\theoremstyle{thm}
\newtheorem{theorem}{Theorem}[section]

\newtheorem{lemma}[theorem]{Lemma}

\newtheorem{remark}[theorem]{Remark}

\theoremstyle{def}
\newtheorem{definition}[theorem]{Definition}

\numberwithin{equation}{section}

\DeclareMathAlphabet{\mathpzc}{OT1}{pzc}{m}{it}

\newcommand{\va}{\varphi}

\def\derat#1{{d \over dt} \hbox{\vrule width0.5pt 
		height 5mm depth 3mm${{}\atop{{}\atop{\scriptstyle t=#1}}}$}}

\newcommand{\vf}{\mathtt{V}}
\newcommand{\dd}{{\tt D}}

\newcommand{\fs}[1]{\mathbbm {#1}}
\newcommand{\eu}[1]{\EuScript {#1}}

\newcommand{\set}[1]{\left\{#1\right\}}

\newcommand{\rr}{\mathbb{R}}
\newcommand{\nn}{\mathbb{N}}

\newcommand{\id}{\mathrm{Id}}
\newcommand{\pr}{\Pi}

\newcommand{\mc}[1]{\mathtt{MC}^{#1}}
\newcommand{\bl}[1] {\mathbf {#1}}

\DeclareMathOperator{\Img}{Img}

\DeclareMathAlphabet\EuScript{U}{eus}{m}{n}
\SetMathAlphabet\EuScript{bold}{U}{eus}{b}{n}
\newcommand\opn{\ensuremath{\mathrel{\mathpalette\opncls\circ}}}

\newcommand{\opncls}[2]{
	\ooalign{$#1\subseteq$\cr
		\hidewidth\raisefix{#1}\hbox{$#1{\stylefix{#1}#2}\mkern2mu$}\cr}}
\def\raisefix#1{
	\ifx#1\displaystyle
	\raise.39ex
	\else
	\ifx#1\textstyle
	\raise.39ex
	\else
	\ifx#1\scriptstyle
	\raise.275ex
	\else
	\raise.150ex
	\fi
	\fi
	\fi
}
\def\stylefix#1{
	\ifx#1\displaystyle
	\scriptstyle
	\else
	\ifx#1\textstyle
	\scriptstyle
	\else
	\ifx#1\scriptstyle
	\scriptscriptstyle
	\else
	\scriptscriptstyle
	\fi
	\fi
	\fi
}
\DeclareFontFamily{U}{mathx}{\hyphenchar\font45}
\DeclareFontShape{U}{mathx}{m}{n}{
	<5> <6> <7> <8> <9> <10>
	<10.95> <12> <14.4> <17.28> <20.74> <24.88>
	mathx10
}{}
\newcommand{\fr}{Fr\'{e}chet }

\newcommand{\Id}{\mathop\mathrm{Id}\nolimits}

\renewcommand{\emptyset}{\varnothing}

\newcommand{\s}{\mathbf{S}}

\makeatletter
\@namedef{subjclassname@2020}{%
	\textup{2020} Mathematics Subject Classification}
\makeatother
\begin{document}

\title{Geometry via sprays on Fr\'{e}chet Manifolds}

\author{Kaveh Eftekharinasab}
\address{Algebra and Topology Department,  Institute of Mathematics of National Academy of Sciences of Ukraine, Tereshchenkivska st. 3,  01024, Kyiv, Ukraine}
\email{kaveh@imath.kiev.ua}
\thanks {}





\begin{abstract}
We construct connection maps and linear symmetric connections on tangent and second-order tangent bundles for \fr manifolds using the notion of a spray. For these manifolds, we characterize linear symmetric connections on tangent bundles in terms of bilinear symmetric mappings associated with sprays.  We also provide an alternative characterization of these connections using tangent structures. Furthermore, we prove that a bijective correspondence exists between linear symmetric connections on tangent bundles and  sprays.
\end{abstract}

\maketitle
\section{Introduction}
The geometry of \fr manifolds, especially higher-order structures on tangent bundles, has seen remarkable advancements in recent years, see for instance, \cite{a,a2,a3, a5,dod,dod2,dod3,dod4,su1}.
A significant  challenge in handling these manifolds is that the general linear group of a \fr space is not a Lie group. Consequently, it cannot serve as a structure group for bundles. Another major issue is that the space of continuous linear mappings between \fr spaces does not possess a \fr space structure. Therefore, defining geometrical objects involving this space, such as connection maps, is not a straightforward generalization of the Banach case.

An approach to resolve these issues involves the use of the projective limit technique as employed in the referenced papers. In this method, we consider \fr manifolds that are projective limits of Banach manifolds. This allows the construction of objects such as structure groups and connection maps using projective limits of their corresponding Banach counterparts.

This approach has been successful in deriving numerous valuable results, as outlined in the aforementioned papers. However, in cases where a structure group is not required (e.g., for covariant differentiation, geodesics, etc.), a  question arises: are there alternative methods to address these problems? This is particularly relevant because, in the projective limit approach, constructing all objects through projective limit systems is highly technical and not always accessible.
In this work, we aim to develop an alternative approach that avoids relying on general linear groups for constructing geometric objects. Specifically, we will investigate the use of the simpler concept of a spray, which has received relatively little attention in the context of \fr manifolds.

In the effort  to define differentiable vector bundles on \fr manifolds without relying on a structure group, a fundamental question arises: how do we define the differentiability of transition mappings, given that they take values in a general linear group? In \cite{hamilton}, Hamilton presented one such definition, while Neeb proposed an alternative definition in \cite{neeb}. Furthermore, a more recent definition was introduced by Gl\"{o}ckner in \cite{gl}.  In this work, we adopt the definition from \cite{neeb}, although the other definitions could also be applied. 

The starting point is to prove that a given spray uniquely determines a connection map, a concept introduced by Dombrowski \cite{dom}. This is accomplished in Theorem \ref{th:con}, which stands as our foundational result, linking a spray with other objects, such as connections on tangent and second-order tangent bundles, facilitating further development. In our approach, it is important to note that the definition of a connection map does not involve Christoffel symbols, nor does it rely on a space of continuous linear mappings between \fr spaces. 
Indeed, we characterize a connection map  through the associated bilinear map of a given spray. Furthermore, a connection map inherits its class of differentiability from the corresponding spray.

Employing parallel transport is a widespread method for defining connections. However, this approach faces limitations in the context of \fr manifolds, since, in general, parallel transports may not exist. Consequently, in our setting, a notable challenge in the concept of a connection lies in its determination from a spray.

In Theorem \ref{th:ch}, we prove that there is a bijective correspondence between sprays and  symmetric linear connections by applying connection maps associated with sprays. Furthermore, in Theorem \ref{th:noncon}, we  prove   the same result using the notion of a tangent structure. 
An analogous  result was originally established by Ambrose, Palais, and Singer \cite{amb} for finite-dimensional smooth manifolds and has since been extended to various contexts using different techniques; see for example, \cite{mb,mb2,im}.

One of our concerns is the geometry of second-order tangent bundles, due to their significance and applications in studying ordinary differential equations (ODEs) on manifolds. A pivotal question related to these bundles is  determining the conditions under which they exhibit a vector bundle structure.  In Theorem \ref{th.iso}, we show that the presence of a spray guarantees the existence of a vector bundle structure on these bundles.
Furthermore, in Theorem \ref{th:ind}, we show that a linear symmetric connection on a tangent bundle induces a linear symmetric connection on the second-order tangent bundle, and vice versa.

It is worth mentioning that  for a category of \fr manifolds known as $ \mc{k} $-\fr manifolds, which notably includes manifolds of Riemannian metrics, the systematic use of sprays has been employed to study geodesics in \cite{k3}. 
This has motivated us to further explore the application of sprays in the study of \fr manifolds.
In the category of $ \mc{k} $-\fr manifolds, the notion of connection maps was introduced in \cite{k1} and   used to define a covariant derivative, which played a key role in proving the Morse-Sard-Brown theorem for these manifolds in \cite{k6}. Additionally, connection maps were employed in \cite{k1} to construct a vector bundle structure on second-order tangent bundles.
However, the notion of connection maps introduced in this paper is not only applicable to $ \mc{k} $-\fr manifolds but is also more general than the previously defined one. As a result, our definition of connection maps extends and enhances the results related to connection maps in \cite{k1} and \cite{k6}.
\section{Prerequisites}\label{sec:1}
We  employ the concept of differentiable maps, known as $C^k$-maps in the Michal–Bastiani sense or Keller's $C_c^k$-maps. This choice is made as it avoids introducing any topology on spaces of continuous linear maps.

By $U \opn \mathsf{T}$, we mean that $U$ is an open subset of the topological space $\mathsf{T}$.
\begin{definition}[\cite{neeb}, Definition I.2.1]
	Let $\fs{E}$ and $\fs{F}$ be locally convex spaces, $ \va: U \opn \fs{E}  \to  \fs{F}$ a mapping. Then the  derivative
	of $\va$ at $x$ in the direction $h$ is defined as 
	$$ \dd\va(x)h := \derat0 \va(x + t h) 
	= \lim_{t \to 0} {1\over t}(\va(x+th) -\va(x)) $$
	whenever it exists. The function $\va$ is called differentiable at
	$x$ if $\dd \va(x)h$ exists for all $h \in \fs{E}$. It is called 
	continuously differentiable, if it is differentiable at all
	points of $U$, and the mapping
	$$ \dd \va : U \times \fs{E} \to \fs{F}, \quad (x,h) \mapsto \dd \va(x)h $$
	is continuous. It is called a {\it $C^k$-map}, $k \in \nn \cup \{\infty\}$, 
	if it is continuous, the iterated directional derivatives 
	$ \dd^{j}\va(x)(h_1,\ldots, h_j)
	$
	exist for all integers $j \leq k$, $x \in U$ and $h_1,\ldots, h_j \in \fs{E}$, 
	and all maps $\dd^j \va : U \times \fs{E}^j \to F$ are continuous.
	Occasionally, we write $ \va' =\dd \va $, and refer to $C^{\infty}$-maps as being smooth.
\end{definition}
The differentiability in the sense of the following definition makes it possible to define a differentiable vector bundle without
requiring a structure group. 
\begin{definition}[\cite{neeb}, Definition II.3.1]\label{def:neeb}
	Let $ \fs{M} $ be a $ C^k, k\geq 1, $ locally convex manifold, and $ \mathrm{Diff}(\fs{M}) $ be the group of
	diffeomorphisms of $ \fs{M} $. Further, let $ \fs{N} $ be a $ C^k$ locally convex manifold. Although, in general, $ \mathrm{Diff}(\fs{M} ) $ does not possess a natural
	Lie group structure,  a map $ \va: \fs{N} \to \mathrm{Diff}(\fs{M})$ is said to be $ C^k$, if the following map 	is of class $ C^k$:
	\begin{equation*}
		\widehat{\va} : \fs{N} \times \fs{M} \to \fs{M} \times \fs{M}, \quad (n,x) \mapsto (\va(n)(x),\va^{-1}(n)(x)).
	\end{equation*}
\end{definition}

This definition has been employed in the development of the Lie theory for Lie groups modeled on locally convex spaces in \cite{neeb}.
\begin{definition}[\cite{neeb}, Definition I.3.8]\label{def:neebvec}
	Let $\fs{M}$ be a $ C^k $-Fr\'{e}chet manifold modeled on a \fr space $\fs{F}$, where $ k\geq1 $, and $\fs{E}$ be another \fr space. A $C^k$-{vector bundle} of type $\fs{E}$ over $\fs{M}$ is a triple $(\pr, \fs{V}, \fs{E})$, consisting a $C^k$-\fr manifold $\fs{V}$, a $ C^k $-map $\pr: \fs{V} \to \fs{M}$, and a \fr space $\fs{E}$, possessing the following properties:
	\begin{enumerate}[label={\bf (VB\arabic*)},ref=VB\arabic*]
		\itemindent=12pt
		\item $ \forall m \in \fs{M} $, the fiber $ \fs{V}_m := \pr^{-1}(m) $ is a \fr space 
		isomorphic to $\fs{E} $.
		\item  $ \forall m \in \fs{M}$, there is an open neighborhood $U$ such that a diffeomorphism
		\begin{equation*}
			\upphi_U:\pr^{-1}(U) \to U \times \fs{E} 
		\end{equation*}
		can be established, where $\upphi_U = (\pr \rvert_U, \uppsi_U)$ and $\uppsi_U : \pr^{-1}(U) \to \fs{E}$ is linear on each $\fs{V}_m$ for $m \in U$.	
	\end{enumerate}
	We then call $ U $ a trivializing subset of $ \fs{M} $ and $ \upphi_U $ a bundle chart. If $ \upphi_U $ and $ \upphi_V $ are two bundle charts and $ U \cap V \neq \emptyset $, then we obtain a diffeomorphism
	\begin{equation*}
		\upphi_U \circ  \upphi_V^{-1}: (U \cap V) \times \fs{E} \to (U \cap V) \times \fs{E}
	\end{equation*}
	of the form $ (x, v) \mapsto (x, \uppsi_{VU} (x)v) $. This leads to a map
	\begin{equation*}
		\uppsi_{UV} : U \cap V \to \bl{GL}(\fs{E})
	\end{equation*}
	for which it does not make sense to speak about smoothness because  the general linear group $ \bl{GL}(\fs{E}) $ is not a Lie group. Nevertheless, $ \uppsi_{UV} $ is of class $ C^k $ in the sense of Definition \ref{def:neeb}, as the map
	\begin{gather*}
		\widehat{\uppsi_{UV}}: (U \cap V) \times \fs{E} \to (U \cap V) \times \fs{E} \\
		(x,v) \mapsto
		(\uppsi_{UV}(x)v,\uppsi_{UV}(x)^{-1}v) =(\uppsi_{UV}(x)v,\uppsi_{VU}(x)v)
	\end{gather*}
	is of class $ C^k $. 
\end{definition}
Sprays were developed for the category of  $ \mc{k} $-\fr manifolds in \cite{k3}, where another type of differentiability, namely $ \mc{k} $-differentiability, is applied. However, as the type of differentiability does not impact the definitions and results required in this paper, we adopt them without providing the proofs.
Throughout, we 
assume that $\fs{M}$ is a $ C^k $-manifold modeled on a Fr\'{e}chet space $\fs{F}$, $ k\geq4 $, and $ \fs{M} $ admits a partition of unity of class \(C^k\).

Let $ \pr_{\fs{M}}: T\fs{M} \to \fs{M} $ and $ \pr_{T\fs{M}}: T(T\fs{M}) \to T\fs{M} $ be the tangent and double tangent bundles, respectively.
We denote the tangent map of a function $ \va $  defined on $\fs{M}$ 
by $ T\va  $  (or $ \va_{*} $) and use subscripts to indicate the fibers of bundles and the restrictions of maps to those fibers.

A $ C^{k-2}$-mapping $\vf : T\fs{M} \to T(T\fs{M})$  such that
\[
T\pi_{\fs{M}} \circ \vf =\Id_{T\fs{M}}
\]
is called
a second-order $ C^{k-2} $-vector field. 
Assume that  $ s $ is a fixed real number, and define the mapping
\begin{equation*}
	L_{T\fs{M}}: T\fs{M} \to T\fs{M},\quad v \mapsto sv.
\end{equation*}
Then,  $ (L_{T\fs{M}})_*: T(T\fs{M}) \to T(T\fs{M}) $ is the  induced map on $ T(T\fs{M}) $ and
\begin{equation*}
	(L_{T\fs{M}})_* \circ L_{T(T\fs{M})}= L_{T(T\fs{M})} \circ (L_{T\fs{M}})_*,
\end{equation*}
which follows from the linearity of $ L_{T\fs{M}} $ on each fiber.

A second-order $ C^{k-2} $-vector filed $\s : T\fs{M} \to T(T\fs{M})$  is said to be a {spray} if 
it satisfies the following condition:
\begin{enumerate}[label=$ \bl{(SP\arabic*)} $,ref=SP\arabic*]
	\itemindent=10pt
	\item \label{eq:sp1} $\s(sv) = (L_{T\fs{M}})_*(s\s(v))$ for
	all $ s \in \rr $ and $ v \in T\fs{M} $.
\end{enumerate}
A manifold that possess a $ C^k $-partition of unity admits a spray of class \(C^{k-2}\).  Important examples are Lindel\"{o}f manifolds modelled on nuclear \fr spaces, cf.
\cite[Theorem 16.10]{km}.

Let $ U \opn \fs{F} $, so that $ T(U) = U \times \fs{F} $, and
$ T(T(U)) = (U \times \fs{F} ) \times (\fs{F} \times \fs{F}) $. 
In this context, the representations of $ L_{TU} $ and $ (L_{TU})_*$ in the chart are given by the
following maps:
\begin{equation*}
	L_{TU}:(x,v) \mapsto (x,sv) \quad \text{and}\quad (L_{TU})_*:(x,v,u,w)\mapsto (x,sv,u,sw).
\end{equation*}
Therefore, we get
\begin{equation*}
	L_{T(TU)} \circ (L_{TU})_*(x,v,u,w)=(x,sv,su,s^2w).
\end{equation*}
In order to avoid confusion, if necessary, in a chart $ U $, we index objects by $ U $ to indicate their local representations.
Let $ {\vf}_U= ({\vf}_{U,1},{\vf}_{U,2}): (U \times \fs{F} ) \to \fs{F} \times \fs{F}$
be a local representation of $ \vf $, where each $ {\vf}_{U,i} $ maps $ U \times \fs{F} $
to $ \fs{F} $ with $ {\vf}_{U,1}(x,v)=v $. Then, $ \vf $ is a spray if and only if, for
all $ s \in \rr $, the following condition holds:
\begin{equation}\label{eq:1}
	{\vf}_{U,2}(x,sv)=s^2{\vf}_{U,2}(x,v).
\end{equation}
Thus, we observe that \eqref{eq:sp1} (in addition to being a second-order vector field), 
simply means that $ {\vf}_{U,2} $ is homogeneous of degree 2 in the
variable $ v $.  It follows that $ {\vf}_{U,2} $ is a quadratic
map in its second variable, and more precisely, this quadratic map is given by
\begin{equation*}
	{\vf}_{U,2} (x,v) = \dfrac{1}{2}\dd_2^2 {\vf}_{U,2} (x, 0)(v,v).
\end{equation*}
Thus, the spray is induced by a symmetric bilinear map given at each point
$ x $ in a chart by
\begin{equation}\label{eq:sp3}
	\mathbb{B}(x)= \dfrac{1}{2}\dd_2^2 {\vf}_{U,2} (x, 0),
\end{equation}
where $ \dd_2^2 $ is the second partial derivative with respect to the second variable.
Conversely, suppose a map $ x \in U \mapsto \mathbb{B}(x) $ is given, where $\mathbb{B}(x): \fs{F} \times \fs{F} \to \fs{F} $ is a symmetric bilinear $ C^{k-2} $-map. 
For
each $ v, w \in \fs{E} $, the value of $  \mathbb{B}(x) $ at $ (v, w) $ is denoted by $ \mathbb{B}(x; v, w) $ or
$ \mathbb{B}(x)(v, w) $. Define $ {\vf}_{U,2}(x, v) = \mathbb{B}(x; v,v) $. Then, $ {\vf}_{U,2} $ is quadratic in its
second variable, and a $ C^{k-2} $-spray over $ U $ can be represented by the map $ {\vf}_{U,2} $ defined by
\begin{equation*}
	{\vf}_{U,2}(x, v) = (v, (\mathbb{B}(x; v, v)) = (v, {\vf}_{U,2}(x, v)).
\end{equation*}
The mapping $ \mathbb{B} $ is called the symmetric bilinear map 
associated with the spray.

Next, we will introduce a covariant derivative; however, it's important to note that this definition is not the most general one available. There exist other definitions that are more general in nature. However, this definition is seamlessly incorporated for the applications of sprays.

Let $ \mathsf{V}^{k}(\fs{M}) $ and $ \mathcal{E}^k(\fs{M}) $ denote the sets of all $C^{k}$-vector fields and $C^k$-real-valued maps on $ \fs{M} $, respectively. Let $X$ and $Y$ belong to $ \mathsf{V}^{k}(\fs{M}) $, and $ [X,Y] $ represent the bracket product. A covariant derivative $ \nabla $ is an $\rr$-bilinear map
\begin{equation*}
	\begin{array}{cccc}
		\nabla : \mathsf{V}^{k}(\fs{M}) \times \mathsf{V}^{k}(\fs{M}) \to \mathsf{V}^{k}(\fs{M}),\quad
		(X,Y) \to \nabla_XY
	\end{array}
\end{equation*} 
such that for all $ \varphi \in \mathcal{E}^k(\fs{M})$ and $ X,Y \in \mathsf{V}^{k}(\fs{M}) $ the following hold: 
\begin{enumerate} [label={\bf (CD\arabic*)},ref=CD\arabic*]
	\itemindent=14pt
	\item \label{eq:cd2} $ \nabla_{\varphi X}Y = \varphi \nabla_{ X}Y$,
	\item \label{eq:cd3} $\nabla_{ X}(\varphi Y) = (\mathcal{L}_{X}\varphi)Y + \varphi \nabla_{ X}Y$,
	\item \label{eq:cd4} $\nabla_XY - \nabla_YX = [X,Y]$.
\end{enumerate}
The last condition is imposed to ensure the existence of a bijective correspondence between sprays and covariant derivatives, see \cite[VIII, $ \S 2 $ ]{lang}.

As proved in\cite{k3}, for a given spray $ \s $  on $ \fs{M} $, there exists a unique covariant
derivative $ \nabla $ (of class $ C^{k-2} $) such that in a chart $ U $, the covariant derivative is expressed as follows:
\begin{equation}\label{eq:cd1}
	(\nabla_XY)_U(x) = Y'_U(x)X_U(x)-\mathbb{B}_U(x)(X_U(x),Y_U(x)).
\end{equation}
\section{Associated Connection Maps}
In this section, we prove our departure key theorem, Theorem \ref{th:con}. This theorem asserts that a spray on a manifolds induces a unique connection map in the sense of \cite{eli} and \cite{vilms}. 

Consider the tangent bundle $ \pr_{\fs{M}}: T\fs{M} \to \fs{M} $ and the double tangent bundles $ \pr_{T\fs{M}}: T(T\fs{M}) \to T\fs{M} $. If  $(U,\phi)$  is a local trivialization  for $ \fs{M} $, then $ \phi(U) \times \fs{F} $ is the corresponding chart for $ T\fs{M} $, and $ (\phi(U) \times \fs{F}) \times (\fs{F} \times \fs{F}) $ represents the corresponding chart for $ T(T\fs{M}) $.

Let $ \pr_{\fs{M}}^*T\fs{M} $ be the pullback bundle induced by $ \pr_{\fs{M}} $, $ \pr_{\fs{M}}^*: \pr_{\fs{M}}^*T\fs{M} \to T\fs{M}$  its projection,  and $ (\phi(U) \times \fs{F}) \times \fs{F} $ 
the corresponding chart for $ \pr_{\fs{M}}^*T\fs{M} $.

In charts, the tangent map $ T\pr_{T\fs{M}} $ takes the following form: $$ (\phi{(U)} \times \fs{F}) \times (\fs{F} \times \fs{F}) \to  \phi{(U)} \times \fs{F}, \quad T\pr_{TU}(u,x,y,z) = (u,y).$$
Furthermore, the following diagram is commutative: 
\begin{equation} \label{dig:1}
	\begin{tikzcd}
		{(\phi{(U)}\times\mathbb{F})\times (\mathbb{F} \times \mathbb{F})} && {\phi{(U)}\times\mathbb{F}} \\
		{\phi{(U)}\times \mathbb{F}} && \phi{(U)}
		\arrow["{T\pr_U}", from=1-1, to=1-3]
		\arrow["{\pr_{U}}", from=1-3, to=2-3]
		\arrow["{\pr_{U}}", from=2-1, to=2-3]
		\arrow["{\pr_{TU}}"', from=1-1, to=2-1]
	\end{tikzcd}
\end{equation}

Define the mapping $ \mathcal{H}: T(T\fs{M}) \to \pr_{\fs{M}}^*T\fs{M} $ such that 
\[T\pr_{T\fs{M}} = \pr_{\fs{M}}^* \circ \mathcal{H}.\]
Then, the diagram in \eqref{dig:1} gives rise to the following diagram:
\begin{equation*}
	\begin{tikzcd}
		{(\phi{(U)} \times \mathbb{F})\times \mathbb{F} \times \mathbb{F}} && {(\phi{(U)} \times \mathbb{F})\times \mathbb{F}  } && {\phi{(U)} \times \mathbb{F}} \\
		{\phi{(U)}\times\mathbb{F}} && {\phi{(U)}\times\mathbb{F}} && \phi{(U)}
		\arrow["{\mathrm{Id}}", from=2-1, to=2-3]
		\arrow["{\pr_{U}^*}", from=1-3, to=2-3]
		\arrow["{\pr_{TU}}", from=1-1, to=2-1]
		\arrow["{\mathcal{H}}_U", from=1-1, to=1-3]
		\arrow["{\pr_{U}}", from=2-3, to=2-5]
		\arrow["{\pr_{U}^*}", from=1-3, to=1-5]
		\arrow["{\pr_{U}}", from=1-5, to=2-5]
	\end{tikzcd}
\end{equation*}
with $ \mathcal{H}_U(u,x,y,z) = (u,x,y)$ and $ \pr_{U}^*(u,y,z)=(u,z) $. 
Let
$ \overline{\mathcal{H}} = \pr_{\fs{M}}^* \circ \mathcal{H}$, then
\[
\overline{\mathcal{H}} =T\pr_{T\fs{M}}. 
\]
Consider another chart $ V $ and an isomorphism $ \va : U \to V $. For a given spray $ \s $, we require a transformation rule for a change of charts under $ \va $. Locally, the map $ T\va $ takes the form:
\begin{equation*}
	\upphi: \phi({U}) \times \fs{F}  \to \fs{F} \times \fs{F}, \quad \upphi(x,v)=(\va(x),\va'(x)v).
\end{equation*}
Then, the change of chart for $ TT\va $ is given  by
\begin{gather*}
	\widehat{\upphi}: \phi{(U)} \times \fs{F} \times \fs{F} \times \fs{F} \to \phi(V) \times \fs{F} \times \fs{F} \times \fs{F}, \quad \widehat{\upphi}=(\upphi, \upphi').
\end{gather*}
The derivative  $ \upphi' $ is given as follows:
\[
\upphi'(x,v)= \begin{bmatrix}
	\va'(x) & 0 \\
	\va''(x)v & \va'(x)
\end{bmatrix}\mathbf{v}, \quad
\mathbf{v}=\begin{bmatrix} u \\ v \end{bmatrix}, \; u,v \in \fs{F}.
\]
Thus,
\begin{equation}
	\widehat{\upphi}(x, v, u, w) \mapsto \big(\va(x), \va'(x)u, \va'(x)v, \va''(x)(u,v) + \va'(x)w \big).
\end{equation}
As mentioned in Section \ref{sec:1}, in a chart $ U $, the local representative of $ \s $ can be expressed as $ {\s_U}=({\s_{U,1}}, {\s_{U,2}}) $, where $ {\s_{U,1}}(x,v)=v $ and $ {\s_{U,2}} $ satisfies Equation \eqref{eq:1}. Similarly, in a chart $ V $, the local representatives of $ \s $ are denoted as ${\s_V}=({\s_{V,1}}, {\s_{V,2}}) $. These representatives satisfy the following equation:
\begin{equation}
	\s_V(\va(x),\va'(x)v) = \big(\va'(x)v, \va''(x)(v,v)+\va'(x)\s_{U,2}(x,v)\big),
\end{equation} 
since $ w=  \s_{U,2}(x,v)$, and $ \s_{U,1}(x,v)=\s_{V,1}(x,v)= v $, we can set $ u=v $. Thus,
\begin{gather}\label{gat:1}\nonumber
	\s_{V,2}\big(\va(x), \va'(x)v \big) = \va''(x)(v, v)+ h'(x)\s_{U,2}(x,v), \\ 
	\mathbb{B}_U\big(\va(x);\va'(x)v,\va'(x)w \big)=\va''(x)(v,w)+\va'(x)\mathbb{B}_V\big(x;v,w \big). 
\end{gather}
\begin{lemma}\label{lem:1}
	Let $ \s $ be a given spray on $ \fs{M} $, and let $ \fs{B} $ be the symmetric bilinear map associated with $ \s $.
	Then, there exists a unique vector bundle morphism 
	\begin{equation} \label{eq:l}
		\mathcal{L}: T(TM) \to \pr_{\fs{M}}^*T\fs{M}
	\end{equation}
	over $ T\fs{M} $
	such that over a chart $ U $, we get
	\begin{equation}
		\mathcal{L}_U(x,v,u,w) = (x,v,w-\fs{B}_U(x;v,u)).
	\end{equation} 
	Furthermore, there exists a unique vector bundle morphism $ \overline{\mathcal{L}} $ such that
	\begin{equation}\label{eq:3}
		\overline{\mathcal{L}}_U(x,v,u,w) = (x,w-\fs{B}_U(x;v,u)).
	\end{equation} 
\end{lemma}
\begin{proof}
	Let $ V $ be another chart and $ \va : U \to V $ an isomorphism. As explained earlier, the transformation rule for $ (TT\va)_U $ under a change of chart is given by
	\begin{gather*}
		\widehat{\upphi}: \phi(U) \times \fs{F} \times \fs{F} \times \fs{F} \to \phi(U) \times \fs{F} \times \fs{F} \times \fs{F}\\  \widehat{\upphi}(x, v, u, w) \mapsto \big(\va(x), \va'(x)u, \va'(x)v, \va''(x)(u,v) + \va'(x)w \big).
	\end{gather*}
	Thus,
	\begin{equation}
		\mathcal{L}_V \circ \widehat{\upphi}(x, v, u, w) =\big(\va(x), \va'(x)u, \va'(x)w \big).
	\end{equation}
	Because $ \va''(x)(u,v) $ drops in the fourth coordinates. Therefore, the family of mappings $ {\mathcal{L}_U} $
	defines a vector bundle morphism over $ T\fs{M} $, and its local expression reveals that it is indeed a vector bundle isomorphism over $ U $. This concludes the proof of the first part. 
	To prove the second part let 
	\begin{equation}\label{eq:lbar}
		\overline{\mathcal{L}} = \pr_{\fs{M}}^* \circ \mathcal{L}
	\end{equation}
	Then, it is a vector bundle morphism satisfying \eqref{eq:3}.
\end{proof}
\begin{theorem}\label{th:con}
	Let $ \s $ be a given spray on $ \fs{M} $, with its associated covariant derivative $ \nabla $.
	Then there exists a unique vector bundle morphism 
	\begin{equation}
		K : T(T\fs{M}) \to T\fs{M}
	\end{equation} 
	such that $ \nabla = K \circ T $, and $ \forall X,Y \in \mathsf{V}^{k}(\fs{M}) $ 
	the following diagram is commutative:
	\[\begin{tikzcd}
		{T\mathbb{M}} & {} & {T(T\mathbb{M})} \\
		{\mathbb{M}} && {T\mathbb{M}}
		\arrow["TX", from=1-1, to=1-3]
		\arrow["K", from=1-3, to=2-3]
		\arrow["{\nabla_YX}", from=2-3, to=2-1]
		\arrow["Y"', from=2-1, to=1-1]
	\end{tikzcd}\]
\end{theorem}
\begin{proof}
	The sought-after morphism is, in fact, $ \overline{\mathcal{L}} $, defined by \eqref{eq:lbar} and \eqref{eq:l}, the existence and uniqueness of which were established in Lemma \ref{lem:1}. In a chart $ U $, the expression for $ K = \overline{\mathcal{L}} $ takes the following form:
	\begin{gather*}
		K_U(x,v): \fs{F} \times \fs{F} \to \fs{F}, \quad
		(u,w) \mapsto w-\fs{B}_U(x;v,u)
	\end{gather*}
	that fulfills the conditions of the theorem. 
\end{proof}
The expression $ \nabla = K \circ T  $ indicates that $ K $ is  a connection map in the sense defined  by Eliasson \cite{eli} and Vilms \cite{vilms}. We call it a connection map associated with the spray. 
This connection map finds complete characterization through the associated bilinear map  $ \fs{B} $. The differentiability of the spray implies that $ K $ is of class $ C^{k-2} $. 

The linearity of $ \fs{B} $ with respect to the second variable implies that $ K $ is linear on the fibers of $ T(T\fs{M}) $, and therefore $ K $ is a linear connection map. Furthermore, the symmetry of the connection map $ K $ in variables $ u $ and $ v $ results from the symmetric nature of $ \fs{B} $.
\begin{remark}
	We define a connection map through the associated covariant of a spray. A natural question arises: can a connection map uniquely determine a spray? Specifically, we inquire whether a covariant derivative alone can determine a spray. In \cite[VIII, $ \S 2 $]{lang}, it was highlighted that in the case of Banach manifolds if a manifold admits a cutoff function and we employ a torsion-free covariant derivative, leading to the existence of the symmetric bilinear mapping $ \fs{B}_U $ for each chart $ U $, then a bijective correspondence emerges between sprays and covariant derivatives.
	
	The situation remains consistent for \fr manifolds, since we apply torsion-free covariant derivatives. Regarding cutoff function, \fr manifolds, especially nuclear manifolds, that admit a differentiable partition of unity, possess cutoff functions.
	Thus, for our purposes, it suffices to assume that manifolds admit differentiable partitions of unity of class at least  \(C^4\). 
\end{remark}
Next, because of the significance of second-order tangent bundles (see \cite{a,dod3}) and the feasibility of addressing them through our approach, we will delve into a detailed examination. Initially, we demonstrate that a connection map induces a vector bundle structure on the second-order tangent bundle.

Recall that the second-order tangent bundle $ T^2\fs{M} $ is the set of all 2-jets of $ \fs{M} $.
The space $ T^2\fs{M} $ has the fiber bundle structure over $ \fs{M} $. Its bundle projection $ \pr^2_{\fs{M}} : T^2\fs{M} \to \fs{M}$ is defined by $ \pr_{\fs{M}}^2(j^2_x(\gamma))=x$, where  $ \gamma : \rr \to \fs{M} $
is a $ C^2 $-mapping  such that $ \gamma(0)=x $. However, it is well known that, in general, $ T^2\fs{M} $ is not a vector bundle.

Let $ \fs{N} $ be a $ C^k $-\fr manifold,  and let $ \va : \fs{M} \to \fs{N} $ be a $ C^k $-mapping, $ k\geq 2 $.
The second-order tangent map $ T^2\va = T^2M \to T^2N $ is defined by $ j^2_x\gamma \mapsto j^2_{\va(x)} (\va \circ \gamma) $. 
Despite being a fiber bundle morphism, $ T^2\va $ is not always linear on the fibers of $ T^2\fs{M} $ and $ T^2\fs{N} $ due to the involvement of second-order derivatives.
Now, we prove that the associated connection map of a spray induces a vector bundle structure on $ \fs{M} $.
\begin{theorem}\label{th.iso}
	Let $ \s $ be a given spray on $ \fs{M} $, with $ K $ being its associated connection map. Then, $ K $ not only induces a vector bundle structure on $ T^2\fs{M} $ over $ \fs{M} $ but also leads to an isomorphism between this vector bundle and the direct sum vector bundle $ T\fs{M} \oplus T\fs{M}$.
\end{theorem} 
\begin{proof}
	The proof is identical to the one in \cite[Theorem 4.1]{k1}. Just, in our setting, we need to consider  the following mapping:
	\begin{equation}\label{eq:iso}
		\Upsilon \coloneq (\pr_{T\fs{M}}, \overline{\mathcal{H}},  K) : T(T\fs{M}) \to T\fs{M} \oplus T\fs{M} \oplus T\fs{M}
	\end{equation}
	to obtain an isomorphism of fiber bundles over $ \fs{M} $.
\end{proof}
Conversely, the existence of a vector bundle structure on $ T^2\fs{M} $ implies the presence of a connection map.
Let $ \mathbf{V}T^2\fs{M} $ be the vertical subbundle of $ T(T^2M) $. The connection map can be obtained by identifying $ \mathbf{V}T^2\fs{M} $ with the pullback of $ T^2\fs{M} $ over itself, denoted as $ (\pr^{2})^*T^2 \fs{M}$, and applying the natural mapping of the pullback, which maps $ (\pr^{2})^*T^2 \fs{M}$ onto $ T^2 \fs{M} $.

In \cite{dod}, it was noted that the induced vector bundle structure on second-order tangent bundles is contingent upon the selection of a connection map. Nevertheless, it was proved that this vector bundle structure remains invariant under conjugate connection maps concerning diffeomorphisms of a manifold. We encounter a similar situation in our context. For the sake of completeness, we will briefly discuss this matter, as the arguments are nearly identical.

We will adopt the concept of the conjugacy of connection maps as introduced in \cite{vas}. Consider two sprays $ \s_1 $ and $ \s_2 $ on $ \fs{M} $, with their associated connection maps being $ K_1 $ and $ K_2 $, respectively. Let $ \upmu $ be a diffeomorphism of $ \fs{M} $. The connection maps $ K_1 $ and $ K_2 $ are called as $ \upmu $-conjugate (or $\upmu$-related) if they satisfy the following condition:
\begin{equation}\label{eq:conj}
	T\upmu \circ K_1 = K_2 \circ T(T \upmu).
\end{equation}
Let $ (U,\phi) $ and $ (V,\psi) $ be local trivializations of $ \fs{M} $, and $ \upmu_{UV} \coloneq \psi \circ \upmu \circ \phi^{-1} $ the local representation of $ T^2\upmu $. Then, locally Equation \eqref{eq:conj} becomes
\begin{gather}\label{eq:conjloc}
	\dd \upmu_{UV}(\phi(x))\big(-\fs{B}_U(\phi(x))(u,u) \big) = \\ \nonumber
	\dd^2 \uptau_{UV}\big(\phi(x)\big)(u,u)-\fs{B}_V \big(\upmu_{UV}(\phi(x)) \big) \Big(\dd \upmu_{UV}(\phi(x))(u), \dd \upmu_{UV}(\phi(x))(u)\Big),
\end{gather}
for every $ (x,u)  \in U \times \fs{F}$.
We require the local representations of $ \upmu $ and its corresponding tangent map $ T\upmu $, as well as the local representations of the connection maps. Detailed computations can be found in \cite[$ \S 1.5.6 $]{dod2}, and we will omit them here. It is important to note that in \cite[$ \S 1.5.6 $]{dod2}, the local representations of connection maps are provided in terms of Christoffel symbols $ \Gamma $, and we must apply the equality
\begin{equation}
	\Gamma(x)(u,v) = - \fs{B}(x)(u,v).
\end{equation}
wherever necessary.
Again, by using the same computations as presented in \cite[$ \S 8.3 $]{dod2}, we determine the  local representation of $ T^2\upmu $ as follows:
define the mapping  
\begin{gather*}
	\lambda_{U}: \pr_{\fs{M}}^2(U) \to U \times \fs{F} \times \fs{F} \\
	\lambda_{U} (j^2_x \gamma) = \Big(x, (\phi \circ \gamma)'(0), (\phi \circ \gamma)''(0) 
	-\fs{B}(\phi(x)) \big( (\phi \circ \gamma)'(0), (\phi \circ \gamma)'(0)\big)
	\Big),
\end{gather*}
for every $ x \in  (\pr_{\fs{M}}^2)^{-1}(U)$. Then $ (U, \phi, \Phi_U^2) $ is the local trivializations of
$ T^2\fs{M} $. Here, the diffeomorphism $ \Phi_U^2: (\pr_{\fs{M}}^2)^{-1}(U) \to \phi(U) \times \fs{F} \times \fs{F} $ is defined by
$$
\Phi_U^2 \coloneq (\phi \times \id_{\fs{F}} \times \id_{\fs{F}}) \circ \lambda_U.
$$
Then the local representation turns into  
\begin{gather}\label{eq:find}
	\Big (\Phi^2_{V, \upmu(x)} \circ T^2_x \upmu \circ (\Phi^2_{U, (x)})^{-1} \Big)(h,k) = \\ \nonumber
	= \Big(\dd \upmu_{UV}(y)(h), \dd \upmu_{UV}(y)(k)+ \dd \upmu_{UV}(y)\fs{B}_U(y)(h,h)
	+ \dd^2 \upmu_{UV}(y)(h,h)- \\ \nonumber
	-\fs{B}_V (\upmu_{UV}(y)) \big( \dd \upmu_{UV}(y)(h), \dd \upmu_{UV}(y)(h)
	\big)
	\Big),
\end{gather}
for every $ (y,h,k) \in \phi(U) \times \fs{F} \times \fs{F} $ and $ x = \phi^{-1}(y) $. 
If $ K_1 $ and
$ K_2 $ are $ \upmu $-conjugated, then by combining \eqref{eq:find} and \eqref{eq:conjloc} we get
\begin{gather}\label{eq:fin}
	\Big (\Phi^2_{V, \upmu(x)} \circ T^2_x \upmu \circ (\Phi^2_{U, (x)})^{-1} \Big)(h,k) 
	= \big( \dd \upmu_{UV}(\phi(x))(h), \dd \upmu_{UV}(\phi(x))(k)
	\big)
\end{gather}
which implies that $ T^2\upmu : T^2 \fs{M} \to T^2\fs{M} $ is linear. Moreover, it implies that
$$
x \mapsto \Phi^2_{V, \upmu(x)} \circ T^2_x \upmu \circ (\Phi^2_{U, (x)})^{-1}
$$
is differentiable, therefore, $ (T^2\upmu, \upmu) $ is a vector bundle isomorphism.
Thus, the vector bundle structures induced by $ K_1 $ and $ K_2 $ on $ T^2\fs{M} $ are isomorphic.
\section{Associated Linear Symmetric Connection}
In this section, our initial concern is to prove the existence of an injective correspondence between sprays and linear symmetric connections.
We will adapt the notion of nonlinear connection due to Barthel \cite{ba} and Vilms \cite{vilms}, originally formulated for Banach manifolds.

Henceforth, for the sake of simplicity, we assume that  $\fs{M}$ is a  smooth Fr\'{e}chet manifold. 

Define the mapping 
\begin{gather}
	\pr_{\fs{M}}!: T(T\fs{M}) \to \pr_{\fs{M}}^*T\fs{M} \\ \nonumber
	\pr_{\fs{M}}^*(v_x)=(x, T\pr_{\fs{M},x}(v_x), \quad \forall v_x \in T_x(T\fs{M})).
\end{gather}
The mapping $ \pr_{\fs{M}}! $ is smooth and onto with 
$ \ker \pr_{\fs{M}}! = \ker T\pr_{\fs{M}} = \mathbf{V}T\fs{M}$ (the vertical subbundle of $ T(T\fs{M}) $).
Consequently, we have the following exact sequence of vector bundles over $ T\fs{M} $:
\begin{equation}\label{eq:exact}
	0 \to \mathbf{V}T\fs{M} \xrightarrow{\eu{I}} T(T\fs{M})\xrightarrow{\pr_{\fs{M}}!} \pr^*_{\fs{M}} T\fs{M} \to 0,
\end{equation}
where $ \eu{I} $ is the canonical inclusion. A smooth nonlinear connection on $ T\fs{M} $ is a splitting $ \eu{C} $ on the left
of the exact sequence \eqref{eq:exact}. It means that $ \eu{C} : T(T\fs{M}) \to \mathbf{V}T\fs{M}$ is a map such that
\begin{equation*}
	\eu{C} \circ \eu{I} =\Id_{\mathbf{V}T\fs{M}}. 
\end{equation*}
The horizontal bundle $ \mathbf{H}T\fs{M}$ is defined as the kernel $\ker \eu{C},  \mathbf{H}T\fs{M} = \ker \eu{C}$. As in the Banach case (see \cite{lang}), we can use similar arguments to prove that in a presence of a connection  we have $ T(T\fs{M})= \mathbf{H}T\fs{M} \oplus \mathbf{V}T\fs{M} $. Conversely, the existence of a subbundle $ \mathbf{H}T\fs{M} $ satisfying
$ \mathbf{H}T\fs{M} \oplus \mathbf{V}T\fs{M} $ implies the existence of a nonlinear connection on $ T\fs{M} $.
A connection  $ \eu{C} $ is called linear, if in addition, 
\begin{equation*}
	\eu{C}: T(T\fs{M}) \to \mathbf{V}T\fs{M} \to T(T\fs{M}) 
\end{equation*}
is $ (T\pr_{\fs{M}},T\pr_{\fs{M}}) $ fiberwise linear. 

It is well known that $ T(T\fs{M}) $ possesses two vector bundle structure as a bundle over $ T\fs{M} $:
one induced by $ T\pr_{T\fs{M}} $, and the other with  $ T\fs{M} $ as the base manifold. There is a canonical involution $ \mathtt{Inv} : T(T\fs{M}) \to  T\fs{M}$ that interchanges two structures. 
In \cite[Theorem 1]{sym}, it was proved that for a finite dimensional manifold $ M $, the canonical involution $ \mathtt{Inv} $ is an isomorphism between the primarily and the secondary vector bundle structures on $ T{M} $
such that the following diagram commutes:
\[\begin{tikzcd}
	{T(T{M})} & {} & {T(T{M})} \\
	{T{M}} & {} & {T{M}}
	\arrow["{\mathtt{Inv}}", from=1-1, to=1-3]
	\arrow["{T\Pi_{M}}", from=1-3, to=2-3]
	\arrow["{\mathrm{Id}}", from=2-1, to=2-3]
	\arrow["{\Pi_{TM}}"', from=1-1, to=2-1]
\end{tikzcd}\]
The mapping $ \mathtt{Inv} $ is self-inverse, and its fixed points precisely correspond to symmetric vectors ($ v \in T(T{M}) $ is called symmetric if $ T\pr_{T{M}}(v) = \pr_{T{M}}(v) $). We can similarly prove that the assertions made in \cite[Theorem 1]{sym} also hold true for infinite-dimensional \fr manifolds.
We call a connection $ \eu{C} $ symmetric if 
\begin{equation}
	\eu{C} = \eu{C} \circ \mathtt{Inv}.
\end{equation}
In the upcoming theorem, we prove that a spray uniquely determines a linear symmetric connection, and vice versa. The proof of this theorem hinges on establishing the relationship between a connection and the associated connection map of a given spray.
\begin{theorem}\label{th:ch}	
	Suppose $\s  $ is a spray on $ \fs{M} $. Then, there exists a unique linear symmetric connection on $ T\fs{M} $ that is entirely characterized
	by the associated symmetric bilinear mappings of $ \s $. Conversely, if $ \eu{C} $ is a linear symmetric connection on
	$ T\fs{M} $, then there exists a unique spray on $ \fs{M} $ whose associated connection map is determined by $ \eu{C} $.  
\end{theorem}
\begin{proof}
	Let $ K: T(T\fs{M}) \to \fs{M} $ be the associated connection map of $ \s $.  Consider a chart $ U $ of $ \fs{M}$, where $ \fs{B}_U $ is the associated symmetric bilinear mapping over $ U $. Let $ K_U(x,v,u,w) = (x,w-\fs{B}_U(x;v,u))$ be the local representation of $ K $ in $ U $. 
	Define locally the mapping
	\begin{gather}
		\eu{K}: \mathbf{V}T\fs{M} \to T\fs{M},\\ \nonumber
		\eu{K}(x,u,0,v)= (x,v), \quad \forall (x,v) \in U \times \fs{F}.
	\end{gather}
	Since $ K $ is a smooth bundle morphism, taking into account the definition of $\eu{K}$, it follows that $ K $ 
	uniquely factors  into $ K= \eu{K} \circ \eu{C} $, where $ \eu{C} : T(T\fs{M}) \to \mathbf{V}(T\fs{M}) $
	is a smooth morphism which is locally given by
	\begin{equation*}
		\eu{C}(x,u,v,w)=(x,u,0,w-\fs{B}_U(x)(v,u))).
	\end{equation*}
	Let $  v= 0$, and thus $ \eu{C} \circ \eu{I} = \Id_{\mathbf{V}T\fs{M}} $, hence $ \eu{C} $ is a connection.
	The linearity of $ \fs{B}_U $ with respect to the second variables implies that $\eu{C} $ is linear on the fibers of $ T(T\fs{M}) $, and therefore $ \eu{C} $ is linear. Also, since $ \fs{B}_U $ is symmetric we have
	\begin{equation*}
		\eu{C}(x,u,v,w) = \eu{C}(x,v,u,w),
	\end{equation*}
	which implies $ \eu{C} \circ \mathtt{Inv} = \eu{C} $,  because the mapping $ \mathtt{Inv} $ switches the middle coordinates.
	Thus, $ \eu{C} $ is symmetric. 
	
	Conversely, suppose that $ \eu{C} $ is a linear symmetric connection. By definition $ \eu{C} $ is the left splitting of the exact sequence \eqref{eq:exact} such that
	\begin{equation*}
		\eu{C} \circ \eu{I} =\Id_{\mathbf{V}T\fs{M}}.
	\end{equation*}
	
	Let $ \mathrm{Pr}_2 $ be the projection of $ T\fs{M} \times T\fs{M} $ onto the second factor. By using the   identification  $ \Bbbk: \mathbf{V}T\fs{M} \simeq T\fs{M} \times  T\fs{M}$, we can define a mapping
	$ \eu{C}_2 :  \mathbf{V}T\fs{M} \to T\fs{M} $ with $ \eu{C}_2 = \mathrm{Pr}_2 \circ \Bbbk   $, such that
	$ (\eu{C}_2, \pr_{\fs{M}}) $ is a vector morphism between $ \mathbf{V}T\fs{M} $ and $ T\fs{M} $.
	Define the mapping 
	\begin{equation*}
		K \coloneq \eu{C}_2 \circ \eu{C} : T(T\fs{M}) \to T\fs{M}.
	\end{equation*}
	It is evident that
	$ v \in T\fs{M} $ is horizontal ($ v \in \mathbf{H}T\fs{M} $)
	if and only if
	\begin{equation}\label{eq:cm}
		\ker K = \mathbf{H}T\fs{M}.
	\end{equation}
	We will determine the local representation of $ K $. Let $ (U,\phi) $ be a local trivialization of $ \fs{M} $.
	The exact sequence $ \eqref{eq:exact} $ locally turns into
	\begin{equation*}
		0 \to \phi(U) \times \fs{F} \times \set{0} \times \fs{F} \to \phi(U) \times \fs{F} \times \fs{F} \times \fs{F} \to
		\phi(U) \times \fs{F} \times \set{0} \times \fs{F} \to 0,
	\end{equation*}
	so that $ \eu{I}: (x,v,0,w) \mapsto  (x,v,0,w)  $,  $ \pr_{\fs{M}}! (x,u,v) \mapsto (x,u,w) $, and
	\begin{equation*}
		\Bbbk : (x,u,0,w) \mapsto (x,u,w).
	\end{equation*}
	Locally, $ \eu{C} $ maps $ (x,u,v,w)$ to $(x,u,0, \eu{C}_U(x,u,v,w)) $,
	where $ \eu{C}_U $ exhibit symmetry with respect to $ u  $ and $ v $ and linearity with respect to $ u, v $ and $ w $.
	
	The equality $ \eu{C} \circ \eu{I} = \Id_{\mathbf{V}T\fs{M}} $ implies that $ \eu{C}_U(x,u,0,w) =w $. Now,   define
	\[
	\widehat{\eu{C}}_U (x,u,v,w)\coloneq \eu{C}(x,u,v,w)-w
	\]
	which does not depend on $ w $, and
	\begin{equation*}
		\widehat{\eu{C}}_U(x,u,v,w_1)-\widehat{\eu{C}}_U(x,u,v,w_2) = \widehat{\eu{C}}_U(x,u,0,w_1-w_2).
	\end{equation*}
	Let $ K_U  $ be the local representation of $ K $ over $ U $. Then,
	\begin{gather*}
		K_U(x,u,v,w) = (\mathrm{Pr}_2 \circ \Bbbk )(x,u,0,\eu{C}_U(x,u,v,w)) = (x, \eu{C}_U(x,u,v,w)).
	\end{gather*}
	However, $ \eu{C}_U(x,u,v,w) = w - \widehat{\eu{C}}_U(x,u,v)  $, where $ \widehat{\eu{C}}_U $
	is symmetric and linear in $ u $ and $ v $. 
	
	Thus, we can write
	\begin{equation*}
		K_U(x,u,v,w) = (x,w- \fs{B}_U(x)(u,v)),
	\end{equation*}
	where $ \fs{B}_U : \phi_U(U) \times \fs{F}  \times \fs{F} \to  \fs{F} $
	is a smooth map, and  symmetric and linear in $ u $ and $ v $.
	This means $ K $ is a connection map, and $ \fs{B}_U  $ induces a spray whose connection map is $ K $.
\end{proof}
\begin{remark}
	As mentioned in the introduction, this result has been obtained in various settings, particularly in \cite{im}, for the space of smooth functions between smooth manifolds using an approach based on the theory of schemes.
	However, the space of smooth functions between smooth manifolds can also be endowed with the structure of \fr manifolds. Thus,  Theorem  \ref{th:ch} can be applied to provide an alternative proof of the result in \cite{im}.
\end{remark}	
Next, we show that any linear connection  on $ T\fs{M} $ induces a linear connection on $ T^2\fs{M} $ and vice versa.
To establish this, we rely on the following lemma, which was proven in \cite{duc} for finite-dimensional manifolds, and the proof in our framework is nearly identical.

Let $ \fs{V}_1 $ and $ \fs{V}_2 $ be smooth vector bundles over $ \fs{M} $, and $ \va: \fs{M} \to \fs{M} $ a $ C^1 $-mapping. A mapping $ \upphi : \fs{V}_1 \to \fs{V}_2 $ is called a $ \va $-morphism if for any 
$ x \in \fs{M} $, there exists a vector bundle chart $ (U, \psi) $ of $ \fs{V}_1 $ at $ x $, a vector bundle chart
$ (V, \uppsi) $ of $ \fs{V}_2 $ at $ \va{(x)} $, and a $ C^1 $-mapping $ \Phi_{VU} : (U \cap \va^{-1}(U)) \times \fs{F} \to \fs{F} $ such that 
$$ \upphi_x \circ \psi^{-1}  = \uppsi^{-1}_{\va(x)} \circ \Phi_{VU}, \quad \forall x \in U \cap \va^{-1}(V).$$
\begin{lemma}\label{lem:2}
	Let $ \upphi: \fs{V}_1 \to \fs{V}_2  $ be a $ \va $-morphism that is isomorphism on fibers. Then, 
	$ T_u \fs{V}_1 \simeq T_{\upphi{(u)}} \fs{V}_2 $ for any $ u \in U $, $ U$ is a chart for $ \fs{M} $. Furthermore, 
	any $($linear symmetric$)$ connection on $ \fs{V}_2 $ induces a $($linear symmetric$)$ connection on $ \fs{V}_2 $.
\end{lemma}
\begin{proof}
	From the identity $ T\pr_{\fs{V}_2} \circ T \upphi = T \va \, \circ T \pr_{\fs{V}_2} $ and the assumption that $ \upphi $ is an isomorphism on fibers, it follows that $ T_u \fs{V}_1 \simeq T_{\upphi{(u)}} \fs{V}_2 $ for any $ u \in U $.
	Now, let $ \eu{C}_2 $ be the left splitting that determines the (linear symmetric) connection on $ \fs{V}_2 $. Let
	\begin{equation*}
		\eu{C}_1 \coloneq (T\phi\mid_{\mathbf{V}\fs{V}_1})^{-1} \circ \eu{C}_2 \circ T \upphi ,
	\end{equation*}
	then it is a left splitting
	for the exact sequence associated with $ \fs{V}_1 $, i.e., a (linear symmetric) connection on $ \fs{V}_1 $. 
\end{proof}
\begin{theorem}\label{th:ind}
	Suppose $\s  $ is a spray on $ \fs{M} $. Any linear symmetric connection on $ T\fs{M} $ induces a linear symmetric connection on the vector bundle $ T^2\fs{M} $, and vice versa.
\end{theorem}
\begin{proof}
	By Theorem \ref{th.iso}	the spray $ \s $ induces a vector bundle structure on $ T^2\fs{M} $. Let $ K $ be the associated connection map of $ \s $. 
	
	Suppose that  $ \eu{C} $ is a linear symmetric connection on $ T\fs{M} $.
	The restriction of $ \eu{C} $ to $ T^2\fs{M} $ ($ T^2 \fs{M} $ is a submanifold of $ T(T\fs{M}) $) is a morphism
	of the vector bundles $ T^2\fs{M} $ and $ T\fs{M} $ that is isomorphism on fibers. Therefore, Lemma
	\ref{lem:2} implies that  $ \eu{C} $ induces a linear symmetric connection on $ T^2\fs{M} $. 
	
	Conversely, let $ \mathcal{C} $ be a linear connection on $ T^2\fs{M} $, and define the mapping  
	\begin{equation*}
		\mathcal{C}_1 \coloneq \mathcal{C} \circ \pr_{T^2\fs{M}} : T( T^2\fs{M}) \to T\fs{M}.
	\end{equation*}
	Consider the diffeomorphism $ \Upsilon : T(T\fs{M}) \to 
	T\fs{M} \oplus T\fs{M} \oplus T\fs{M}$ defined in Equation \eqref{eq:iso} which induces an isomorphism on fibers.
	The second-order tangent bundle $ T^2\fs{M} $ is a submanifold of $ T(T\fs{M}) $ consisting of vectors $ v $ such that 
	\[
	\pr_{T\fs{M}}(v) = T\pr_{\fs{M}}(v).
	\]
	Since \(\Upsilon  |_{T\fs{M} \oplus T\fs{M} }\)
	is  the natural isomorphism  onto
	$ (\pr_{T\fs{M}} \oplus K \oplus T\pr_{\fs{M}})(T^2\fs{M}) $, it follows that
	\begin{equation*}
		\Upsilon^{-1} \circ (\pr_{T\fs{M}} \oplus K \oplus T\pr_{\fs{M}})(T^2\fs{M}) = T\pr_{\fs{M}} \oplus K (T^2\fs{M}),
	\end{equation*}
	and, therefore,
	\[
	\mathbf{K} \coloneq T\pr_{\fs{M}} \circ K : T^2\fs{M} \to \fs{M}
	\]
	is a diffeomorphism, and determines the vector bundle charts. 
	
	Let $ \varPi : T\fs{M} \oplus T\fs{M} \to T\fs{M} $ be the projection. Then, $  \varPi\circ \mathbf{K} \circ \mathcal{C}_1  $ gives the morphism of the bundles $ T^2\fs{M} $ and $ T\fs{M} $ 
	which is an isomorphism on fibers. Therefore, Lemma
	\ref{lem:2} implies that  $ \mathcal{C} $ induces a connection on $ T\fs{M} $. 
\end{proof}
Next, we define a tangent structure as a straightforward generalization from finite-dimensional manifolds and use it to characterize  connections.

The vertical space $ \mathbf{V}_xT\fs{M} $ at $ v \in T \fs{M} $ is naturally identified with $ T_x\fs{M} $, $ x=\pr_{\fs{M}}(v) $. Define the mapping
\begin{gather}\label{eq:verlift}
	\mathtt{Ver}_v: T_x\fs{M} \to \mathbf{V}_vT\fs{M}  \\ \nonumber
	\mathtt{Ver}_v(w) = \derat0 v + t w, \quad  w \in T_x\fs{M}.
\end{gather}
The mapping $ \mathtt{Ver}_v $ is a linear isomorphism for any $ v \in T\fs{M} $ and is called a vertical lift. Since
$ \ker K = \mathbf{H}T\fs{M} $ (Equation \eqref{eq:cm}), we can easily show that for each $ v \in T_x\fs{M} $, the connection map $ K $ maps $ T_vT\fs{M} $ to $ T_{\pr_{\fs{M}}(v)} \fs{M}$, and satisfies
\begin{equation}\label{eq:inco}
	K (\mathtt{Ver}_v(w)) = w, \quad \forall w \in T_{\pr_{\fs{M}}(v)} \fs{M}.
\end{equation}
Since, for $ h \in T\fs{M} $, the mapping $ \Pi_{\mathbb{M}} \mid_{\mathbf{H}_hT\fs{M}} $ is isomorphism, and
\begin{equation*}
	T\pr_{\fs{M},x} : T_hT\fs{M} \to T_x\fs{M} (x =\pr_{T\fs{M}}(h))
\end{equation*}
is an epimorphism, then
$ T\pr_{\fs{M},h} \mid_{\mathbf{H}_hT\fs{M}} : \mathbf{H}_hT\fs{M} \to T_x\fs{M} $ is an isomorphism. The inverse of the latter isomorphism $ \mathtt{Hor}_h : T_x\fs{M} \to \mathbf{H}_hT\fs{M} $ is called horizontal lift induced 
by the connection.
Now, define the mapping
\begin{gather}
	\jmath : \pr^*_{T\fs{M}} \to \mathbf{V}T\fs{M} \\ \nonumber
	\jmath(x_1,x_2) = \mathsf{Ver}_{x_1}(x_2),\; x_1,x_2 \in T\fs{M},\; \pr_{\fs{M}}(x_1) = \pr_{\fs{M}}(x_2).
\end{gather}
For $ u \in T\fs{M} $, define the mapping $ \eu{J}_u : T_uT\fs{M} \to  \mathbf{V}_uT\fs{M}$ by 
\begin{equation*}
	\eu{J}_u \coloneq \eu{I}_u \circ \jmath_u \circ \pr_{\fs{M},u}!.
\end{equation*}
The mapping $ \eu{J} $, called the tangent structure, can be considered as a  $ \mathcal{E}^{\infty}(\fs{M}) $-linear mapping from $ \mathsf{V}(T\fs{M}) $ to itself, defined by 
\begin{equation*}
	\eu{J}(V)(u) = \eu{J}_u(V_{\pr_{T\fs{M}}(u)})\quad \text{for } V \in  \mathsf{V}(T\fs{M}), \; u \in T\fs{M}.
\end{equation*}
Directly, from the definition it follows that $ \eu{J}^2 =0 $, and 
\begin{equation*}
	\Img (\eu{J}) = \ker (\eu{J}) =  \mathsf{Sec} (\mathbf{V}T\fs{M}) (\text{the space of smooth sections}).
\end{equation*}
The horizontal and vertical lifts are related by 
\begin{equation}\label{eq:impor}
	\eu{J}_u \circ \mathtt{Hor}_u = \mathtt{Ver}_u,
\end{equation} 
that is the following diagram is commutative: 
\[\begin{tikzcd}
	{T_uT\fs{M}} & {} & {T_uT\fs{M}} \\
	& {T_{\Pi_\fs{M}(u)}\fs{M}}
	\arrow["{\eu{J}_u}", from=1-1, to=1-3]
	\arrow["{\mathtt{Ver}_u}", from=1-3, to=2-2]
	\arrow["{T\Pi_{\fs{M},u}}"', from=1-1, to=2-2]
\end{tikzcd}\]
Compose  Equation \eqref{eq:impor} from the left by $K_u $ so $ K_u \circ \eu{J}_u \circ \mathtt{Hor}_u =  K_u \circ\mathtt{Ver}_u  $. By \eqref{eq:inco}, the right side is the identity $ \id_{T_{\pr_{\fs{M}}(v)} \fs{M}} $. 
Then, since $ \mathtt{Hor}_u $ is a linear isomorphism, it follows that $ K_u \circ \eu{J}_u $ is its inverse, and hence we obtain
\begin{equation}\label{eq:last}
	K_u \circ \eu{J}_u = T\pr_{\fs{M},u}.
\end{equation}
Let $ \mathsf{V}(\fs{M}) $ and $ \mathsf{V}(T\fs{M}) $ be the modules of smooth vector fields on $ \fs{M} $ and
$ T\fs{M} $, respectively. Let $ \eu{C} $ be a linear symmetric connection on $ T\fs{M} $. It induces a morphism
$ \widehat{\eu{C}} : \mathsf{V}(\fs{M}) \to \mathsf{Sec} (\mathbf{V}T\fs{M}) $.
Define the mapping
\begin{gather}
	\mathtt{Vp} : \mathsf{V}(T\fs{M})  \to \mathsf{V}(T\fs{M}) \\ \nonumber
	\mathtt{Vp}(v) \coloneq
	\begin{cases}
		1 & \text{if } \widehat{\eu{C}}(v) \;\text{is vertical}\\
		0 & \text{if } v \;\text{is horizontal}.
	\end{cases}
\end{gather}
The definition of $ \mathtt{Vp} $ directly implies that $\mathtt{Vp}^2=\mathtt{Vp}  $.

Now, define
$ \mathtt{Hp} \coloneq \id_{\mathbf{V}T\fs{M}} -  \mathtt{Vp} $ so that we have $\mathtt{Hp}^2=\mathtt{Hp}  $.
We call $ \mathtt{Vp}  $  ($ \mathtt{Hp}  $) the vertical  (horizontal) projector associated with $ \eu{C} $.
\begin{theorem} \label{th:noncon}
	Any linear symmetric connection on the tangent bundle determines a  connection map associated with a spray, and vice versa. 
\end{theorem}
\begin{proof}
	Let $ \eu{C} $ be a linear symmetric  connection, and let $ u \in T\fs{M} $. The mapping
	\begin{equation*}
		\eu{J}_u\mid_{\mathbf{H}_uT\fs{M}}: \mathbf{H}_uT\fs{M} \to \mathbf{V}_uT\fs{M}
	\end{equation*}
	is an isomorphism. Let $ \eu{L}_u $ be its inverse, and define its extension to $ T_uT\fs{M} $ by $ \eu{L} \coloneq \eu{L}_u \circ \eu{C}_u $.
	Directly from the definition, it follows that $ \eu{L}^2 = 0  $ and $ \Img \eu{L} = \ker \eu{L} = \mathbf{H}T\fs{M} $.
	Moreover, $ \eu{L} \circ \eu{J} =  \mathtt{Hp}$ and $ \eu{J} \circ \eu{L} = \mathtt{Vp} $. Now define the linear mapping
	\begin{gather*}
		K_u : T_uT\fs{M} \to T\pr_{\fs{M}(u)}\fs{M} \\
		K_u = T\pr_{\fs{M}(u)} \circ \eu{L}_u.
	\end{gather*}
	This is a connection map, to proof this we need to show that $ K_u \circ \eu{J}_u = T\pr_{\fs{M}(u)} $. Since $ \eu{L}_u \circ \eu{J} = (\mathtt{Hp})_u $
	and 
	\[
	T\pr_{\fs{M}(u)} \circ  (\mathtt{Hp})_u = T\pr_{\fs{M}(u)},
	\]	
	it follows that $ K_u \circ \eu{J}_u = T\pr_{\fs{M}(u)} $.
	
	Conversely, let $ K $ be a connection map associated with given spray \(\s\), and let $ u \in T\fs{M} $. Then, $ K_u \circ \eu{J}_u = T\pr_{\fs{M},u}$ (Equation \eqref{eq:last}).
	Thus, since $ T\pr_{\fs{M},u} $ is an epimorphism, $ K_u  $ is also an epimorphism. Let $ \mathbf{H}_u T\fs{M}\coloneq  \ker K_u $
	and $ \mathbf{V}_u \coloneq  \ker \eu{J}_U $. From $ K_u \circ \eu{J}_u = T\pr_{\fs{M},u}$  it follows that 
	\[
	\mathbf{H}_u T\fs{M} \cap \mathbf{V}_u T\fs{M} = \set{0},
	\]
	and hence 
	\[
	\mathbf{H}_u T\fs{M} \oplus \mathbf{V}_u T\fs{M} = T_uT\fs{M}.
	\]
	Therefore, there exists a connection which inherits linearity and symmetry   from
	the associated bilinear map   associated with  the spray \(\s\) .
\end{proof}
    
\bibliographystyle{amsplain}

\begin{thebibliography}{10}
 \bibitem{a3}
M. Aghasi and A. Suri, \textit{Splitting theorems for the double tangent bundles of \fr manifolds}, 
Balkan J. Geom. Appl. \textbf{15} (2010), 2, 1-13.
\bibitem{a2}
M. Aghasi, A. Bahari, C. Dodson,  G.  Galanis and A. Suri, 
\textit{Second order structures for sprays and connections on \fr manifolds}, http://arxiv.org/abs/0810.5261v1	
\bibitem{a}	
M. Aghasi, A. Bahari, C. Dodson,  G.  Galanis and A. Suri,  
\textit{Infinite-dimensional second order ordinary differential equations via $ T^2M $}, Nonlinear Analysis: Theory, Methods and Applications \textbf{67} (2007), 10, 2829-2838.
\bibitem{a5}	
M. Aghasi, A. Bahari, C. Dodson,  G.  Galanis and A. Suri, 
\textit{Conjugate connections and differential equations on infinite dimensional manifolds}. In : J. A. Alvarez Lopez and E. Garcia-Rio (Eds.), \textit{Differential geometry:  Proc VIII international colloquium on differential geometry}, pp. 227-236.
Santiago de Compostela,  2008.
\bibitem{amb}  
W. Ambrose, R. S. Palais and I. M. Singer, \textit{Sprays}, An. Acad. Brasil. Ci. \textbf{32} (1960), 163-178.   
\bibitem{ba}
W. Barthel, \textit{Nichtlineare Zusammenhänge und deren Holonomiegruppen}, J. Reine Angew. Math. \textbf{212} (1963), 120-149. 
\bibitem{mb}
M. Bunge,  \textit{On a synthetic proof of the Ambrose-Palais-Singer theorem for infinitesimally linear spaces}. Cahiers de Topologie et G\'{e}om\'{e}trie Différentielle Cat\'{e}goriques, \textbf{28} (1987), 2,  127-142. 
\bibitem{mb2}
M. Bunge and P. Sawyer, \textit{On connections, geodesics and sprays in synthetic differential geometry}, 
Cahiers de Topologie et G\'{e}om\'{e}trie Différentielle Cat\'{e}goriques, \textbf{25} (1984), 3, 221-257.
\bibitem{dod3}	
C. Dodson  and G. Galanis, \textit{Second order tangent bundles of infinite dimensional manifolds},
J. Geom.  Phys.  \textbf{52} (2004), 2, 127-136.
\bibitem{dod4}	
C. Dodson, G. Galanis, E. Vassiliou, \textit{A generalized second-order frame bundle for \fr manifolds},
J. Geom.  Phys. \textbf{55} (2005), 3, 291-305.
\bibitem{dod2}	
C. Dodson  and G. Galanis  and E. Vassiliou E,  \textit{Geometry in a \fr context: A projective limit approach}.  Cambridge University Press, Cambridge, 2015.
\bibitem{dod}
C. Dodson,  G. Galanis  and E. Vassiliou, \textit{Isomorphism classes for Banach vector bundle structures of second tangents}, Math. Proc. Cambridge Philos. Soc \textbf{141}, (2006), 3, 489-496.
\bibitem{dom} 
P. Dombrowski, \textit{On the Geometry of the Tangent Bundle}, {J. Reine Angew. Math.} \textbf{210} (1962), 73-88.
\bibitem{duc} 
T. Duc, \textit{Sur la g\'{e}om\'{e}trie diff\'{e}rentielle des fibr\'{e}s vectoriels}, 
{Kodai math. Sem.
	Rep.} \textbf{26} (1975), 349-408.
\bibitem{k6}
K. Eftekharinasab, \textit{The Morse-Sard-Brown Theorem for Functionals on Bounded-\fr-Finsler Manifolds}, 
Communications in Mathematics,  \textbf{23} (2015), 2, 101-112.	
\bibitem{k1}
K. Eftekharinasab, \textit{Geometry of bounded \fr manifolds},  {Rocky Mountain J. Math.} \textbf{46} (2016), 3, 895-913.
\bibitem{k3}
K. Eftekharinasab  and V. Petrusenko, \textit{Finserian geodesics on \fr manifolds}, 
{Bulletin of the Transilvania University of Barsov. Series III: Mathematics, Informatics, Physics}
\textbf{13} (2020), 1, 129-151.
\bibitem{eli}
H. Eliasson, \textit{Geometry of manifolds of maps}, {J. Differential Geom.} \textbf{1} (1967), 169-194.
\bibitem{sym}
R. Fisher and H. Taquer, \textit{Second order tangent vectors in Riemannian geometry}, {J. Korean Math. Soc.} \textbf{36} (1999), 5, 959-1008.
\bibitem{ga1}
G. Galanis, \textit{Universal connections in \fr principal bundles}, {Period Math. Hung.} \textbf{54}
(2007), 1, 1-13.
\bibitem{gl}
H. Gl{\"{o}}ckner, \textit{Aspects of differential calculus related to infinite-dimensional vector
	bundles and poisson vector spaces}, Axioms, \textbf{11} (2022), 5, 221.
\bibitem{hamilton}
R. Hamilton, \textit{The inverse function theorem of Nash and Moser}, {Bull. Amer. Math. Soc.} \textbf{7}  (1982), 1,
65-222.
\bibitem{lang}
S. Lang, \textit{Fundamentals of differential geometry}. Springer, New York,  1999.
\bibitem{km}
A. Kriegl and P. Michor, \emph{The convenient setting of global analysis}, Mathematical Surveys and Monographs 53, AMS, 1997.
\bibitem{im}
I. Moerdijk and E. Reyes, \textit{On the relation between connections and sprays},
Revista Colombina de Matem\'{a}ticas \textbf{20} (1986), 187-222.
\bibitem{neeb}
K-H. Neeb, \textit{ Towards a Lie theory of locally convex groups}, {Jpn. J. Math.} \textbf{1}  (2006), 2, 291-468.
\bibitem{su1}
A. Suri  and M. Moosaei, \textit{Bundle of frames and sprays for \fr Manifolds}, {Int. Electron. J. Geom.} \textbf{11}
(2018), 1, 1-16. 
\bibitem{vas}
E. Vassiliou, \textit{Transformations of linear connections}, {Period. Math. Hung.} \textbf{13} (1982), 4, 289-308.
\bibitem{vilms}
J. Vilms, \textit{Connections on tangent bundles}, {J. Differential Geom.} \textbf{1} (1967),  235-243.

\end{thebibliography}

\end{document}